\documentclass[12pt]{article}
\usepackage[dvips]{graphicx}
\usepackage{amsmath,amsfonts,amssymb,amsthm,color}
\usepackage{pgf,tikz}

\usepackage[numbers]{natbib}

\newtheorem{theorem}{Theorem}
\newtheorem{lemma}{Lemma}

\newtheorem{proposition}{Proposition}

\newtheorem{example}{Example}

\usepackage{bm}
\bmdefine{\Bt}{t}
\bmdefine{\BX}{X}
\bmdefine{\BY}{Y}
\bmdefine{\BZ}{Z}
\bmdefine{\BB}{B}
\bmdefine{\BM}{M}
\bmdefine{\BD}{D}
\bmdefine{\Bi}{i}
\bmdefine{\Bj}{j}
\bmdefine{\Bk}{k}
\bmdefine{\Bx}{x}
\bmdefine{\By}{y}
\bmdefine{\Bz}{z}
\bmdefine{\Bv}{v}
\bmdefine{\Bw}{w}
\bmdefine{\Bn}{n}
\bmdefine{\Ba}{a}
\bmdefine{\Bb}{b}
\bmdefine{\Bc}{c}
\bmdefine{\Be}{e}
\bmdefine{\Bu}{u}
\bmdefine{\Bp}{p}
\bmdefine{\Bzero}{0}
\bmdefine{\Bone}{1}

\newcommand{\Z}{{\mathbb Z}}

\newcommand{\cF}{{\cal F}}
\newcommand{\cI}{{\cal I}}
\newcommand{\cJ}{{\cal J}}

\newcommand{\cS}{{\cal S}}

\newcommand{\ncolon}{\!:\!}

\newcommand{\paths}{W}

\newcommand{\drb}[2]{\draw[very thick] (V#1)--(V#2);}
\newcommand{\drw}[2]{\draw[very thick, dotted] (V#1)--(V#2);}

\newcommand{\drbl}[2]{\draw [thin](V#1)--(V#2);}
\newcommand{\drwl}[2]{\draw[thin,dotted] (V#1)--(V#2);}

\newcommand{\mytkzini}{
\foreach \i in {1,...,3} {   \foreach \j in {1,...,4}    {
  \path (\j,-\i) coordinate (V\i\j);  \fill (V\i\j) circle (2pt); }}
}

\title{Markov chain Monte Carlo test of toric homogeneous Markov chains}
\author{
Akimichi Takemura\footnote{
Graduate School of Information Science and Technology, 
University of Tokyo}
\footnote{JST, CREST} \ and 
Hisayuki Hara\footnote{
Department of Technology Management for Innovation, 
University of Tokyo}
}
\date{August  2010}

\begin{document}
\maketitle

\begin{abstract}
  Markov chain models are used in various fields, such behavioral sciences 
  or econometrics. 
  Although the goodness of fit of the model is usually assessed by
  large sample approximation, it is desirable to use conditional tests
  if the sample size is not large.
  We study  Markov bases for performing conditional tests of the 
  toric homogeneous Markov chain model, which is the envelope exponential family for the
  usual homogeneous Markov chain model. 
  We give a complete description of a Markov basis for
  the following cases: i) two-state, arbitrary length, ii) arbitrary
  finite state space and length of three.  The general case
  remains to be a conjecture.  We also present a numerical example of 
  conditional tests based on our Markov basis.

\end{abstract}

\section{Introduction}
Consider a Markov chain $X_t$, $t=1,\dots,T\,(\ge 3)$,
over a finite state space $\cS=\{1,\dots,S\}$ ($S\ge 2$).
Let $\pi_i$, $i\in \cS$,  denote the initial distribution of $X_1$ 
and let $p^{(t)}_{ij}=P(X_{t+1}=j \mid X_{t}=i)$ %
denote the
transition probability from time $t$ to $t+1$.
We want to test the null hypothesis
of homogeneity
\begin{equation}
\label{eq:null-hypothesis}
H_0:\  p^{(t)}_{ij} = p_{ij}, \qquad t=1,\dots,T-1.
\end{equation}
This model is used as a standard model in many fields. 
See \citet{Haccou-Meelis} for
behavioral sciences and \citet{Sylvia-Fruhwirth-Schnatter} for econometrics.
Usually the goodness of fit of the model is assessed by
large sample approximation (\citet{anderson-goodman}, \citet{billingsley-1961}).  
However when
the sample size is not large, it is desirable to use conditional tests.
The Markov basis methodology (\citet{diaconis-sturmfels}) is attractive
for performing them.

In this paper we study Markov bases for performing conditional tests
of the null hypothesis of homogeneity.
However the homogeneous Markov chain model (\ref{eq:null-hypothesis})
is a curved exponential family due to 
the constraints $\sum_{j\in \cS} p_{ij}=1$,
$\forall i\in \cS$, and we cannot directly apply  conditional tests 
to $H_0$ in (\ref{eq:null-hypothesis}).  
Instead, by the Markov basis methodology we
can test a larger null hypothesis $\bar H_0$ 
that  $X_t, t=1,\dots,T$, are observations
from a {\em toric} homogeneous Markov chain
(cf.\ Section 1.4 of \citet{ascb}), which is the envelope exponential
family 
(cf.\ \citet{kuchler-sorensen-1996aism} and \citet[Chapter 7]{kuchler-sorensen-book})
of the homogeneous Markov chain model (\ref{eq:null-hypothesis}).  
In Section \ref{subsec:THMC-interpretation} 
we discuss interpretations of the toric homogeneous Markov chain model and
its difference from the usual homogeneous Markov chain model.

In the following for notational simplicity we write
THMC for ``toric homogeneous Markov chain''.
Note that in Section 1.4 of \citet{ascb} they consider
a model without parameters for the initial distribution, whereas 
our THMC model 
contains parameters for the initial distribution.
As an alternative hypothesis to THMC model 
we can take non-homogeneous Markov chain model, which is an
(affine or full) exponential family model which includes THMC model as
a submodel.

In this paper we  derive complete description of Markov bases for THMC model
for the case of $S=2$ (arbitrary $T$) and the
case of $T=3$ (arbitrary $S$). 
For other combinations of $(S,T)$ with small $S$ and $T$, we can
use a computer algebra package 4ti2 (\citep{4ti2}).   
Complete description of Markov bases for general $S$ and $T$ seems
to be a difficult problem at present. 

The organization of this paper is
as follows.  In Section
\ref{sec:preliminaries} we first set up our model 
and then give some preliminary results.
We derive a Markov basis for the case of $S=2$ in Section \ref{sec:two-state} 
and for $T=3$ in Section \ref{sec:three-period}.
In Section \ref{sec:numerical} we present an application of our Markov basis to a
data set. In Section \ref{sec:discussion} we give some discussions on Markov bases for general $S$ and $T$.

\section{Toric homogeneous Markov chain model}
\label{sec:preliminaries}

\subsection{THMC model and its interpretation}
\label{subsec:THMC-interpretation}

Let 
$\omega=(s_1, \dots, s_T)\in \cS^T$ denote an observed path of a Markov
chain. 
Each path is considered as one frequency of an 
$\vert \cS \vert^T$ contingency table.
The usual homogeneous Markov chain model (\ref{eq:null-hypothesis})
specifies the probability of the
path $\omega$ as
\begin{equation}
\label{eq:HMC-prob}
H_0 : \  p(\omega)=\pi_{s_1} p_{s_1 s_2} \dots p_{s_{T-1} s_T},
\end{equation}
where the normalization is 
$\sum_{i\in \cS}\pi_i=1$ and  $\sum_{j\in \cS} p_{ij}=1$, $\forall i\in \cS$.
In THMC model, the normalization is only assumed for the total probability.
In order to avoid confusion, let $\gamma_i$, $i\in \cS$, and $\beta_{ij}$, $i,j\in \cS$, be free nonnegative parameters.  We specify THMC model by
\begin{equation}
\label{eq:THMC-prob}
\bar H_0 :  \ p(\omega)=c \gamma_{s_1} \beta_{s_1 s_2} \dots \beta_{s_{T-1} s_T},
\end{equation}
where $c$ is the overall normalizing constant.
THMC model (\ref{eq:THMC-prob}) is the envelope exponential family 
of the homogeneous Markov chain model  (\ref{eq:HMC-prob}), i.e., the
linear hull  
of logarithms of the probabilities in  
(\ref{eq:HMC-prob}) coincides with the set of logarithms
of probabilities in (\ref{eq:THMC-prob}).
This simply follows from the fact that the linear hull %
of elementwise logarithms of probability vector
\[
\{ (\log p_1, \dots, \log p_S) \mid p_i > 0, i\in \cS, \ \sum_{i\in \cS}p_i=1 \}
\]
is the whole ${\mathbb R}^S$.
The THMC model $\bar H_0$ contains  the usual homogeneous Markov chain model $H_0$
and the difference in degrees of freedom  is $S-1$.

As discussed in 
\citet{kuchler-sorensen-1996aism} 
an envelope exponential family is often difficult to interpret.  For THMC model we give
the following interpretation.   Write
\[
\beta_{ij}=\alpha_i p_{ij},  \qquad \sum_{j\in \cS} p_{ij}=1 .
\]
Then 
(\ref{eq:THMC-prob}) is written as
\[
p(\omega) = c \gamma_{s_1} \alpha_{s_1} p_{s_1 s_2} \dots \alpha_{s_{T-1}}
   p_{s_{T-1,T}} = c \gamma_{s_1}  p_{s_1 s_2} \dots  p_{s_{T-1,T}} \times ( \alpha_{s_1} \dots \alpha_{s_{T-1}}). 
\]
Compared to (\ref{eq:HMC-prob}), $\alpha_i$ specifies a  magnifying ratio
of probability of a path visiting the state $i$.  We can interpret $\alpha_i$ as
an internal growth rate or ``birth rate'' of state $i$, which is not affected by
immigration or emigration.  

As an alternative hypothesis we can take the  non-homogeneous Markov chain model
\begin{equation}
\label{eq:non-homogeneous}
H_1 :\ p(\omega)=\pi_{s_1} p^{(1)}_{s_1 s_2} \dots p^{(T-1)}_{s_{T-1} s_T}.
\end{equation}
By reparametrization this model can be simply written as 
the linearly ordered conditional independence model:
\[
p(\omega)= c\beta^{(1)}_{s_1 s_2} \beta^{(2)}_{s_2 s_3} \dots \beta^{(T-1)}_{s_{T-1} s_T}. 
\]
This is a decomposable
model of a $T$-way contingency table, with its independence graph
given by the following linear tree:
\begin{center}
\begin{tikzpicture}
\draw (0,0) node [above] {$1$} node {$\bullet$} 
-- (1,0) node [above] {$2$} node {$\bullet$}
-- (2.5,0) node [above]{$\dots$}
-- (4,0) node [above]{$T-1$}node {$\bullet$} 
-- (5,0) node [above]{$T$}node {$\bullet$} ;
\end{tikzpicture}
\end{center}
This model is a full exponential family model containing THMC model.

In this paper we consider testing $\bar H_0$ against $H_1$ via Markov basis approach.  
Although by this procedure we are not directly testing $H_0$, it should be noted 
that we can reject $H_0$ if $\bar H_0$ is rejected.  
When $\bar H_0$ is accepted, we further need
to test $H_0$ against $\bar H_0$.  However at present there is no general
procedure for finite sample exact tests  of curved exponential family.

\subsection{Sufficient statistic and moves for THMC model}
\label{subsec:definitions}

Suppose that we observe  $N$ paths $\omega_1,\dots, \omega_N$ from 
a homogeneous Markov chain (\ref{eq:null-hypothesis}).
We write $\paths=\{\omega_1,\dots, \omega_N\}$.  Here multiple
paths are allowed in $\paths$ and hence $\paths$ is a multiset.
Let $x(\omega)$ denote the frequency of the path $\omega$ in $\paths$.
Then the data set is summarized in a $T$-way contingency
table $\Bx=\{x(\omega), \omega\in \cS^T\}$.

Let $x^{t}_{ij}$ denote the number of transitions from $s_t=i$ to $s_{t+1}=j$ in $\paths$ and
let $x^{t}_i$ denote the frequency of the state $s_t=i$ in $\paths$.
In particular $x^1_i$ is the frequency of the 
initial state $s_1=i$.
Let
\[
x^+_{ij}=\sum_{t=1}^{T-1} x^t_{ij}
\]
denote the total number of transitions from $i$ to $j$ in $\paths$.
Under %
$H_0$ the joint probability of $\Bx$ is written as
\begin{align*}
p(\Bx) &= \frac{N!}{\prod_{\omega\in \cS^T} x(\omega)!} \prod_{\omega\in \cS^T}
(\pi_{s_1} p_{s_1 s_2} \dots p_{s_{T-1} s_T})^{x(\omega)}
\\
&= \frac{N!}{\prod_{\omega\in \cS^T} x(\omega)!} \prod_{s\in \cS} \pi_s^{x^1_s} \prod_{i,j\in \cS}
p_{ij}^{x^+_{ij}}.
\end{align*}
Therefore the sufficient statistic for $H_0$ 
is given by the initial frequencies and the frequencies  of transitions
\begin{equation}
\label{eq:sufficient-statistic}
\Bb=\Bb(\Bx)=\{x^1_i , i \in \cS\} \  \cup \ \{x^+_{ij}, i,j\in \cS\}.
\end{equation}

In fact this sufficient statistic for $H_0$ is minimal sufficient and 
it is also the sufficient statistic for THMC model.  This simply reflects
the fact that THMC model is the envelope exponential family of the
homogeneous Markov chain model. 
Since THMC model (\ref{eq:THMC-prob}) is an exponential family with
integral sufficient statistic, we can perform the conditional test of 
$\bar H_0$ %
by the Markov basis methodology.  
Therefore our goal in this paper is to obtain Markov bases for $\bar
H_0$ with various $S$ and $T$. 

If we order paths of $\cS^T$ appropriately and write $\Bx$ as a column
vector according to the order, $\Bb=\Bb(\Bx)$ in
(\ref{eq:sufficient-statistic}) is written in a matrix form 
\[
\Bb = A \Bx , 
\]
where $A$ is an $S(S+1)\times S^T$ matrix consisting of non-negative
integers. 
The set of  all contingency tables 
sharing $\Bb$ is called a {\it fiber} and denoted by 
$\cF_{\Bb}=\{\Bx \in \Z^{S^T}_{\ge 0} \mid A\Bx = \Bb\}$, 
where $\Z_{\ge 0}=\{0,1,\dots\}$.
A move $\Bz$ for THMC model is an integer array satisfying 
$A \Bz =0$. 
Then $\Bz$ is expressed by a difference of two contingency tables $\Bx$
and $\By$ in the same fiber: 
\[
\Bz = \Bx - \By, \qquad  z(\omega)=x(\omega)-y(\omega), \ \omega\in \cS^T.
\]
We write $z^t_i= x^t_i - y^t_i$  and 
$z^t_{ij}=x^t_{ij}-y^t_{ij}$.  
Note that $\sum_{i\in \cS}z^t_i=0$ for all $t$, because
$N=\sum_{i\in \cS}x^t_i=\sum_{i\in \cS}y^t_i$ is the
number of paths. Also since $\Bx$ and $\By$ is in the same fiber
$\sum_{t=1}^{T-1}z^t_{ij}=0$ for all $i,j\in \cS$. 

\begin{example}
For illustration we consider the case of $S=2$ and $T=4$. 
The configuration $A$ is written as
\begin{equation}
\label{eq:s2t4A}
A=
\setlength{\tabcolsep}{2pt}
\begin{tabular}{lcccccccccccccccc}
&\scriptsize{1111} & \scriptsize{1112} & \scriptsize{1121}& \scriptsize{1122} & \scriptsize{1211}& \scriptsize{1212} & \scriptsize{1221} & \scriptsize{1222} & 
\scriptsize{2111} & \scriptsize{2112} & \scriptsize{2121} & \scriptsize{2122} & \scriptsize{2211} & \scriptsize{2212} & \scriptsize{2221} &  \scriptsize{2222}\\
& 1 & 1 & 1 & 1 & 1 & 1 & 1 & 1  & 0 & 0 & 0 & 0 & 0 & 0 & 0 & 0 \\
& 0 & 0 & 0 & 0 & 0 & 0 & 0 & 0  & 1 & 1 & 1 & 1 & 1 & 1 & 1 & 1 \\
``11''\ & 3 & 2 & 1 & 1 & 1 & 0 & 0 & 0  & 2 & 1 & 0 & 0 & 1 & 0 & 0 & 0 \\
``12''\ & 0 & 1 & 1 & 1 & 1 & 2 & 1 & 1  & 0 & 1 & 1 & 1 & 0 & 1 & 0 & 0 \\
``21''\ & 0 & 0 & 1 & 0 & 1 & 1 & 1 & 0  & 1 & 1 & 2 & 1 & 1 & 1 & 1 & 0 \\
``22''\ & 0 & 0 & 0 & 1 & 0 & 0 & 1 & 2  & 0 & 0 & 0 & 1 & 1 & 1 & 2 & 3 \\
\end{tabular}
\end{equation}
The first two rows of $A$ correspond to $x^1_1$ and $x^1_2$.  
The other rows correspond to $x^+_{11}$, $x^+_{12}$, $x^+_{21}$, 
$x^+_{22}$ respectively. 
The columns of $A$ are indexed by the paths $\omega\in \{1,2\}^4$ in
the lexicographic order. 
Note that compared to a similar configuration in Section 1.4.1 of
\cite{ascb}, $A$ in (\ref{eq:s2t4A}) has additional two rows
corresponding to initial frequencies.

An interesting feature of the above $A$ is that it contains identical
columns.   
The columns 1121 and 1211 are identical.
This means that the difference of the two paths
$\Bz =\{z(\omega) \mid \omega \in \cS^T\}$
\[
 z(\omega)= \left\{
 \begin{array}{rl}
  1, & \text{if } \omega = 1121,\\
  -1,& \text{if } \omega = 1211,\\
   0, & \text{otherwise}
 \end{array}
\right. 
\]
forms a degree one move.
We depict the degree one move as
\begin{equation}
\label{movegraph}
\begin{tikzpicture}[baseline=-1.2cm,scale=0.8]
 \foreach \i in {1,2} {   \foreach \j in {1,...,4}    {
 \path (\j,-\i) coordinate (V\i\j);  \fill (V\i\j) circle (1pt); }}
 \draw (V11)--(V12)--(V23)--(V14);
 \draw [dotted] (V11)--(V22)--(V13)--(V14);
 \draw (0,-1) node {1};
 \draw (0,-2) node {2};
 \draw (0,-0.4) node {$s\backslash t$};
 \foreach \j in {1,...,4} { \draw (V1\j)  node [above]  {\j};}
\end{tikzpicture}\;, 
\end{equation}
where a solid line from $(i,t)$ to $(j,t+1)$ represents  
$z^t_{ij}=1$ and a dotted line from $(i,t)$ to $(j,t+1)$ represents  
$z^t_{ij}=-1$. 
We note that the columns 2122 and 2212 are also identical.
We could remove redundant columns from $A$ and only leave a
representative column from the set of identical columns.  However it
seems better to leave identical columns in view of the symmetry in $A$
and we will work with the above form of $A$. 
\end{example}

In the following we often use the graph representation of moves as
(\ref{movegraph}).
We call such a graph a move graph.
A node of a move graph is a pair $(i,t)$ of state $i$ and time $t$ and
an edge from $(i,t)$ to $ (j,t+1)$ represents the value of $z^t_{ij}$.  
If $\vert z^t_{ij} \vert =0$, there is no corresponding edge in the
graph. 
If $\vert z^t_{ij} \vert \ge 2$, we write the value of 
$\vert z^t_{ij} \vert$ beside the edge. 
For example, the move for $S=2$ and $T=4$ such that 
$z^1_{11}=z^3_{22}=-1$, $z^2_{12}=-2$ and 
$z^1_{12}=z^2_{22}=z^2_{11}=z^3_{12}=1$ is represented by
\[
\begin{tikzpicture}[baseline=-1.2cm,scale=0.8]
\foreach \i in {1,2} {   \foreach \j in {1,...,4}    {
 \path (\j,-\i) coordinate (V\i\j); \fill (V\i\j) circle (1pt);}}
 \drwl{11}{12}; \drwl{12}{23}; \drwl{23}{24};
 \drbl{11}{22}; \drbl{22}{23}; \drbl{12}{13}; \drbl{13}{24};
 \draw (2.8,-1.5) node {\scriptsize 2};
\end{tikzpicture}\quad . 
\]
We note that a move graph does not have a one-to-one correspondence to a
move and more than one move have the same move graph in general.

For notational simplicity we denote the edge 
from $(i,t)$ to $(j,t)$ as $t\ncolon ij$.
Given %
a contingency table $\Bx$,
$x^t_i$ is the number of paths passing through the node $(i,t)$ and
$x^t_{ij}$ is the number of paths passing through the edge 
$t\ncolon ij$. 
Consider a partial path $(s_t, s_{t+1}, \dots, s_{t'})$ 
starting at time $t$ and ending at time $t'>t$.
We write this partial path as $t\ncolon s_t s_{t+1} \dots s_{t'}$. 
$x^t_{s_t s_{t+1} \dots s_{t'}}$ 
denotes the frequency of the partial path $t\ncolon s_t s_{t+1} \dots s_{t'}$
in $\paths$.  For a particular path $\omega\in \cS^T$ we say that
$\omega$ passes through $t\ncolon s_t s_{t+1} \dots s_{t'}$
if $\omega$ is of the form  $(*,\dots,*,s_t,s_{t+1},\dots,s_{t'},*,\dots,*)$ where
$*$ is arbitrary.

Let $\Bz = \Bx - \By$. 
If $z^t_i>0$ we say that $\Bx$ dominates $\By$ 
at the node $(i,t)$. Similarly  if $z^t_{ij}>0$ we say that 
$\Bx$ dominates $\By$  in the edge $t\ncolon ij$.
Consider a partial path $t\ncolon s_t \dots s_{t'}$.
If  all of 
$z^t_{s_t,s_{t+1}}$, $z^{t+1}_{s_{t+1},s_{t+2}}$, \dots, 
$z^{t'-1}_{s_{t'-1},s_{t'}}$ are positive, we say that
$\Bx$ dominates $\By$ in the partial path $t\ncolon s_t \dots s_{t'}$.

Given the signs of some $z^t_{ij}$'s, 
we depict the pattern of signs as an undirected graph with thick edges.
Positive edges are depicted by thick solid lines and negative edges by
thick dotted line.  
In the case of $S=2$ and $T=4$, when $z^1_{11}, z^2_{12}$ and $z^3_{22}$
are positive and $z^1_{12}, z^2_{11}, z^2_{22}$ and $z^3_{12}$ are
negative, the corresponding graph is depicted by 
\begin{equation}
\begin{tikzpicture}%
\path (0,-1.5) coordinate (midy);
\draw (0,-0.4) node {$s\backslash t$};
\foreach \i in {1,2} {   \foreach \j in {1,...,4}    {
  \path (\j,-\i) coordinate (V\i\j);  \fill (V\i\j) circle (2pt); }}
\foreach \j in {1,...,4} {
  \draw (\j,-0.4) node {$\j$};}
\draw (0,-1) node {1};
\draw (0,-2) node {2};
\drb{11}{12}; \drb{12}{23}; \drb{23}{24};
\drw{11}{22}; \drw{22}{23}; \drw{12}{13}; \drw{13}{24};
\end{tikzpicture}
\label{eq:ppflip}
\end{equation}
regardless of the signs of $z^1_{22}$, $z^1_{21}$, $z^2_{21}$,
$z^3_{21}$ and $z^3_{11}$ where no edge is depicted. 
We call such graphs edge-sign pattern graphs. 
We note that edges of an edge-sign pattern graph have only the
information on the signs of $z^t_{ij}$ and do not have the
information on the value of $z^t_{ij}$ like a move graph.

For considering our Markov basis it is important to
note that THMC %
model possesses the symmetry with respect
to time reversal.  The following fact is well known, as
discussed in Section 2 of \citet{billingsley-1961}.

\begin{lemma}
\label{lemma:terminal}
For all $\Bx$ in the same fiber $\cF_{\Bb}$ of THMC model, the terminal
 frequencies $x^T_j$, $j\in \cS$, are common.
\end{lemma}

In terms of moves,
the above lemma implies that  %
$z^T_i=0$  
for all $i\in \cS$.
Since $z^1_i=0$, we have 
\begin{equation}
\label{eq:zsum2T1}
\sum_{t=2}^{T-1} z_i^t = 0.
\end{equation}

Lemma \ref{lemma:terminal} shows that given $\Bb$,
the paths of the original homogeneous Markov chain and %
the paths of the time-reversed homogeneous Markov chain has the
same conditional distribution.  This symmetry with respect to time
reversal is important in considering a Markov basis for the present problem.

\subsection{Linearly ordered conditional independence model and crossing
path swapping}

Let $\Bx$ and $\By$ in the same fiber be given.  We call $\Bx$ and
$\By$ {\em edge-wise equivalent} if $z^t_{ij}=0$ for all $t$
and for all $i,j$.  Edge-wise equivalence does not mean that
$\Bx$ and $\By$ are identical as a multiset of $N$ paths.
However consider a non-homogeneous Markov chain
$H_1$ in (\ref{eq:non-homogeneous}), which is a linearly ordered conditional
independence model of $T$-way contingency tables.
Then $x^t_{ij}$, $1\le t\le T-1$ and $i,j\in \cS$, constitute
the sufficient statistic for the model.
Therefore moves for $H_1$ is expressed by the difference of two
edge-wise equivalent paths and their move graphs have no edges.

We note that $H_1$ is a decomposable model for contingency tables.
By \citet{dobra-2003bernoulli} there exists a Markov basis for
this %
model consisting of square-free degree two moves.   
These square-free degree two moves are related  to the idea of %
swapping of two paths meeting (or crossing) at a node.  
Let
$\bar \omega=(s_1, \dots, s_T)$ and
$\bar \omega'=(s_1', \dots, s_T')$ be two paths.  
We say that these two paths meet (or cross) at the node $(i,t)$ if
$i=s_t=s_t'$. If  $\omega$ and $\omega'$ cross 
at the node $(i,t)$, %
consider the swapping of these two paths like 
\begin{align}
&\{\omega,\omega'\}=\{\; (s_1, \dots, s_{t-1},i,s_{t+1},\dots,s_T), \ 
(s_1', \dots, s_{t-1}',i,s_{t+1}',\dots,s_T')\; \} \nonumber\\
& \qquad  \leftrightarrow
\{\; (s_1, \dots, s_{t-1},i,s_{t+1}',\dots,s_T'),
(s_1', \dots, s_{t-1}',i,s_{t+1},\dots,s_T)\; \} :=\{\tilde\omega, \tilde\omega'\}
\label{eq:crossing}
\end{align}
\begin{center}
\begin{tikzpicture}
\draw (4,2) node [above] {$t$};
\draw [thin] (1,0) -- (3,0) 
-- (4,1) node {$\bullet$} node [left]{$i$} 
-- (5,2) -- (7,2) ;
\draw [thin] (1,2) -- (2,1) -- (3,2) 
-- (4,1) node {$\bullet$} node [left]{$i$} 
-- (5,0) -- (7,0) ;
\draw [<->]  (6,0.2) -- (6,1.8);
\end{tikzpicture}
\end{center}
Then the difference $\bm{z}$ of $(\omega,\omega')$ and 
$(\tilde\omega,\tilde\omega')$ 
\[
 z(\omega) = \left\{
 \begin{array}{rl}
  1 & \text{if } \omega = \bar \omega\; \text{or}\; \bar \omega'\\
  -1& \text{if } \omega = \tilde \omega \; \text{or}\; \tilde \omega'\\
  0 & \text{otherwise}
 \end{array}
\right.
\]
forms a move for $H_1$.
We call this move %
a ``crossing path swapping''.
By  \citet{dobra-2003bernoulli} we have the following proposition.

\begin{proposition}
\label{prop:swapping}
The set of crossing path swappings in (\ref{eq:crossing})
constitutes a Markov basis for the linearly ordered conditional independence model.
\end{proposition}

Therefore once $\Bx$ and $\By$ of the same fiber are brought to be
edge-wise equivalent, they can be further brought to be identical as
multisets by crossing path swappings.  Therefore for constructing a
Markov basis for THMC model, we need to find
an additional finite set of moves, which can make two elements of the
same fiber edge-wise equivalent.

Crossing path swapping has the following implication.
Suppose that $\Bx$ dominates $\By$  in the partial path 
$t \ncolon s_t\dots s_{t^\prime}$. 
Then clearly there exists a path $\omega$ in $\Bx$ passing through the
edge $t\ncolon s_t s_{t+1}$.   Similarly there exists a path $\omega'$
in $\Bx$ passing through the edge $t+1\ncolon s_{t+1} s_{t+2}$.  
We can then make the crossing path swapping of these two paths and make 
$\omega$ go through the partial path $t\ncolon s_t s_{t+1} s_{t+2}$.
Repeating this procedure, by a sequence of crossing path swappings of
paths of $\Bx$, we can assume that $\Bx$ contains a path $\omega$ which
goes through the partial path $t\ncolon s_t \dots s_{t^\prime}$.  
We state this as a lemma. 

\begin{lemma}
\label{rem:partial-dominance}
Let $\Bx$ and $\By$ be two elements of the same fiber.
If $\Bx$ dominates $\By$  in the partial path 
$t \ncolon s_t \dots s_{t^\prime}$,
then by crossing path swappings of paths in $\Bx$, we can transform
$\Bx$ to $\tilde \Bx$, such that there exits a path in $\tilde \Bx$ going
through the partial path $t\ncolon s_t \dots s_{t'}$.
\end{lemma}

Note that during the sequence of crossing path swappings, the values of 
$\{z^t_{ij}\}$ stay the same.   
By Lemma \ref{rem:partial-dominance}, we can identify all $\Bx$'s with
the same edge frequencies, whenever  we are concerned about 
decreasing $\sum_{t,i,j} |z^t_{ij}|$ for our distance reduction argument
in Sections  \ref{sec:two-state} and \ref{sec:three-period}.

\subsection{Some properties of moves for THMC model}
\label{subsec:preparations}
In this subsection we present an important class of moves for THMC
model. 
Consider two different states $i_1 \neq i_2$ at time $t$ and two different
states  $j_1\neq j_2$ at time $t+1$.  Let 
$t\ncolon i_1 j_1$  and $t\ncolon i_2 j_2$ be two edges joining these states.
Now consider %
$t' > t$ and two edges
$t'\ncolon i_1 j_2$  and $t'\ncolon i_2 j_1$, which swaps the
transitions $i_1 j_1$, $i_2 j_2$ of time $t$.
Suppose that all of $x^t_{i_1 j_1}$, $x^t_{i_2 j_2}$, 
$x^{t'}_{i_1 j_2}$, $x^{t'}_{i_2 j_1}$ are positive.
Choose (not necessarily distinct) four paths $\omega_1, \omega_2, \omega_3, \omega_4$ from $\Bx$
passing through
$t\ncolon i_1 j_1$, $t\ncolon i_2 j_2$, 
$t'\ncolon i_1 j_2$,  $t'\ncolon i_2 j_1$, respectively:
\begin{align*}
\omega_1 &= (s_{11}\dots,s_{1,t-1},i_1,j_1,
s_{1,t+2},\dots,s_{1T}),
\ 
\omega_2 = (s_{21}\dots,s_{2,t-1},i_2,j_2,
s_{2,t+2},\dots,s_{2T})
\\
\omega_3 &= (s_{31}\dots,s_{3,t'-1},i_1,j_2, s_{3,t'+2},\dots,s_{3T})
, \ \omega_4 = (s_{41}\dots,s_{4,t'-1},i_2,j_1, s_{4,t'+2},\dots,s_{4T}).
\end{align*}
Then we consider swapping the transitions in 
$\{t\ncolon i_1 j_1,t\ncolon i_2 j_2\}$
and those in $\{t'\ncolon i_1 j_2,t'\ncolon i_2 j_1\}$ in $\Bx$ as  
$\{\omega_1, \omega_2, \omega_3, \omega_4\} \leftrightarrow
\{\tilde\omega_1, \tilde\omega_2, \tilde\omega_3, \tilde\omega_4\}$, where
\begin{align*}
\tilde\omega_1 &= (s_{11}\dots,s_{1,t-1},i_1,
j_2,s_{2,t+2},\dots,s_{2T}),
\ 
\tilde\omega_2 = (s_{21}\dots,s_{2,t-1},i_2,
j_1,s_{1,t+2},\dots,s_{1T}),
\\
\tilde\omega_3 &= (s_{31}\dots,s_{3,t'-1},i_1,
j_1, s_{4,t'+2},\dots,s_{4T})
, \ 
\tilde\omega_4 = (s_{41}\dots,s_{4,t'-1},i_2,
j_2, s_{3,t'+2},\dots,s_{3T}). 
\end{align*}
The resulting move  $\Bz=\{z(\omega) \mid \omega \in \cS^T\}$ 
is written as
\[
  z(\omega) = \left\{
 \begin{array}{rl}
  1 & \text{if } \omega = \omega_1,\ldots,\omega_4,\\
  -1& \text{if } \omega = \tilde \omega_1,\ldots,\tilde \omega_4,\\
  0 & \text{otherwise.}
 \end{array}
\right. 
\]
We call  this move a ``2 by 2 swap''.
The corresponding move graph is depicted as follows,
\begin{equation}
\label{eqfig:2by2}
\begin{tikzpicture}%
\path (0,-2.5) coordinate (midy);
\foreach \i in {1,...,4} {   \foreach \j in {1,2}    {
  \path (\j,-\i) coordinate (V\i\j);  \fill (V\i\j) circle (2pt); }}
\foreach \i in {1,...,4} {   \foreach \j in {4,5}    {
  \path (\j,-\i) coordinate (V\i\j);  \fill (V\i\j) circle (2pt); }}
\draw (1,-1) node [left] {$i_1$};
\draw (4,-1) node [left] {$i_1$};
\draw (1,-3) node [left] {$i_2$};
\draw (4,-3) node [left] {$i_2$};
\draw (2,-2) node [left] {$j_1$};
\draw (5,-2) node [left] {$j_1$};
\draw (2,-4) node [left] {$j_2$};
\draw (5,-4) node [left] {$j_2$};
\drbl{11}{22}; \drbl{31}{42}; \drbl{14}{45}; \drbl{34}{25};
\drwl{11}{42}; \drwl{31}{22}; \drwl{14}{25}; \drwl{34}{45};
\end{tikzpicture}
\end{equation}
We denote the set of 2 by 2 swaps sharing the move graph
(\ref{eqfig:2by2}) by  
\begin{equation}
\label{eq:2by2-swap}
\{t\ncolon i_1 j_1,t\ncolon i_2 j_2\} \leftrightarrow 
\{t'\ncolon i_1 j_2,t'\ncolon i_2 j_1\}.
\end{equation}

As noted above, $\omega_1, \omega_2, \omega_3$ and $\omega_4$ are not
necessarily distinct. 
It may be the case that $\omega_1 = \omega_3$, 
$\omega_1 = \omega_4$, 
$\omega_2 = \omega_3$ or $\omega_2 = \omega_4$.
In the 2 by 2 swap in
(\ref{eq:2by2-swap}) and (\ref{eqfig:2by2}), the paths are only
partially specified.   Therefore it corresponds to a set 
of moves in term of fully specified paths.
In particular, depending on which paths are the same, 
a 2 by 2 swap may 
be of degree 2,3, or 4 in terms of the fully specified paths.

The above 2 by 2 swap can be extended  in various ways.  First it 
can be extended to a permutation of more states
\begin{equation}
\label{eq:m-by-m-swap}
\{t\ncolon i_1 j_1, \dots , t\ncolon i_m j_m\}\ \leftrightarrow\ 
\{t'\ncolon i_1 j_{\sigma(1)}, \dots, t'\ncolon i_m j_{\sigma(m)}\},
\end{equation}
where $\sigma$ is an element of the symmetric group $S_m$.  We call this move
an ``$m$ by $m$ permutation''.  The degree of this move is at most
$2m$ in terms of paths.  
The degree is smaller if $\sigma(l)=l$ for some $l$, because swapping 
$t\ncolon i_l j_l$ and $t'\ncolon i_l j_l=t'\ncolon i_l j_{\sigma(l)}$ does not change anything.

Second, it can be extended to swapping of two partial paths
\[
t\ncolon \underbrace{i * \dots * j}_k
\quad \leftrightarrow \quad 
t'\ncolon \underbrace{i * \dots * j}_k  \ , \qquad (t\neq t'),
\]
where `*' is arbitrary.
For example we can swap $t\ncolon 112$ and $t'\ncolon 122$, ($t\neq t'$):
\begin{align}
 &\{ s_1\dots s_{t-1}112s_{t+3}\dots s_T,   s'_1\dots
 s'_{t'-1}122s'_{t'+3}\dots s'_T \}  
 \nonumber \\
 & \qquad \leftrightarrow
 \{ s_1\dots s_{t-1}122s_{t+3}\dots s_T,   s'_1\dots
 s'_{t'-1}112s'_{t'+3}\dots s'_T \}, 
\label{eq:1:1swap}
\end{align}
which can be depicted as 
\begin{equation}
\label{eqfig:11swap}
\begin{tikzpicture}%
\path (0,-1.5) coordinate (midy);
\path (V11) node [above] {$t$};
\path (V13) node [above] {$t'$};
\foreach \i in {1,2} {   \foreach \j in {1,...,5}    {
  \path (\j,-\i) coordinate (V\i\j);  \fill (V\i\j) circle (1pt); }}
\drbl{11}{12}; \drbl{12}{23}; \drbl{13}{24}; \drbl{24}{25};
\drwl{11}{22}; \drwl{22}{23}; \drwl{13}{14}; \drwl{14}{25};
\end{tikzpicture}
\end{equation}
%Note that,  as in the case of a 2 by 2 swap, each of the above extensions 
%corresponds to a set of moves in terms of paths.

\section{Markov basis for two state space case}
\label{sec:two-state}

In this section we give an explicit form of a Markov basis for the case of $S=2$.
The length $T$ of the Markov chain is arbitrary.

\begin{theorem}
\label{thm:s2}
A Markov basis for $S=2$, $T\ge 4$,  consists of the following moves:
\begin{enumerate}
\setlength{\itemsep}{0pt}
\item Crossing path swappings.  
\item Degree one moves.
 \item 2 by 2 swaps of the following form:   
       \begin{tikzpicture}[baseline=-1.2cm,scale=0.7]
	\foreach \i in {1,2} {   \foreach \j in {1,...,4}    {
	\path (\j,-\i) coordinate (V\i\j);  \fill (V\i\j) circle (1pt); }}
	\drbl{11}{12}; \drbl{21}{22}; \drbl{13}{24}; \drbl{23}{14};
	\drwl{11}{22}; \drwl{21}{12}; \drwl{13}{14}; \drwl{23}{24};
       \end{tikzpicture}
\item Moves of the following form: 
 \begin{tikzpicture}[baseline=-1.2cm,scale=0.7]
  \foreach \i in {1,2} {   \foreach \j in {1,...,5}    {
  \path (\j,-\i) coordinate (V\i\j);  }}
  \drbl{11}{12}; \drbl{12}{23}; \drbl{13}{24}; \drbl{24}{25};
  \drwl{11}{22}; \drwl{22}{23}; \drwl{13}{14}; \drwl{14}{25};
 \end{tikzpicture} \  .  \\
      This includes the case of a double transitions in the middle: 
 \begin{tikzpicture}[baseline=-1.1cm,scale=0.7]
  \foreach \i in {1,2} {   \foreach \j in {1,...,4}    {
  \path (\j,-\i) coordinate (V\i\j);  }}
  \drbl{11}{12}; \drbl{12}{23}; \drbl{23}{24};
  \drwl{11}{22}; \drwl{22}{23}; \drwl{12}{13}; \drwl{13}{24};
  \draw (2.8,-1.5) node {2};
 \end{tikzpicture}   \ .
\end{enumerate}
 For $T=3$ a Markov basis consists of the first three types of moves.
\end{theorem}

\begin{proof}
 We use an induction on the length $T$ of the Markov chain combined with
 the distance reduction argument in \citep{takemura-aoki-2005bernoulli}
 for the distance $\vert \Bz \vert := \sum_{t,i,j} |z^t_{ij}|$.

 If $z^1_{ij}=0$, $\forall i,j$, then we can ignore the transitions
 $1\ncolon ij$, $i,j\in \cS$, in both $\Bx$ and $\By$ and regard the
 chain as starting 
 from $t=2$.  
 Similarly if $z^{T-1}_{ij}=0$, $\forall i,j$, then we can regard the
 chain as ending at time $t=T-1$. Therefore we can assume that
 $z^1_{ij}\neq 0$ for some $i,j$ and $z^{T-1}_{ij}\neq 0$ for some
 $i,j$. 
 In view of the symmetry of the model and the fact that
 $z^1_1=z^1_2=0$, we can assume without loss of generality that
 $z^1_{11}>0$ and $z^1_{12}<0$:   
 \[
 \begin{tikzpicture}[baseline=-15mm]
  \foreach \i in {1,2} {   \foreach \j in {1,2}    {
  \path (\j,-\i) coordinate (V\i\j); \fill (V\i\j) circle (2pt); }}
  \draw (V11) node [above] {$1$};
  \draw (V12) node [above] {$2$};
  \drb{11}{12}; \drw{11}{22};
 \end{tikzpicture} 
 \]

 As the initial step of the induction we assume $T=3$.  
 Then by (\ref{eq:zsum2T1}) $z^2_{i}=0$, $i=1,2$.  
 This necessarily implies that the edge-sign pattern %
 of $\Bz$ is of the following form: 
 \[
 \begin{tikzpicture}[baseline=-1.5cm]%
  \foreach \i in {1,2} {   \foreach \j in {1,...,3}    {
  \path (\j,-\i) coordinate (V\i\j); \fill (V\i\j) circle (2pt); }}
  \draw (V11) node [above] {$1$};
  \draw (V12) node [above] {$2$};
  \draw (V13) node [above] {$3$};
  \drb{11}{12}; \drb{21}{22}; \drb{12}{23}; \drb{22}{13};
  \drw{11}{22}; \drw{21}{12}; \drw{12}{13}; \drw{22}{23};
 \end{tikzpicture} 
 \]
 By Lemma \ref{rem:partial-dominance}, we can assume that 
 $\Bx$ contains the path $112$ and $221$. 
 Therefore by adding the 2 by 2 swap 
\[
 \begin{tikzpicture}[baseline=-1.5cm]
  \foreach \i in {1,2} {   \foreach \j in {1,...,3}    {
  \path (\j,-\i) coordinate (V\i\j); \fill (V\i\j) circle (1pt); }}
  \draw (V11) node [above] {$1$};
  \draw (V12) node [above] {$2$};
  \draw (V13) node [above] {$3$};
  \drwl{11}{12}; \drwl{21}{22}; \drwl{12}{23}; \drwl{22}{13};
  \drbl{11}{22}; \drbl{21}{12}; \drbl{12}{13}; \drbl{22}{23};
 \end{tikzpicture}
\]
 to $\Bx$, we can reduce $|\Bz|$ by eight.
 This takes care of the case of $T=3$. 

 Now consider the case of $T \ge 4$.  First we explain our strategy.  
 We claim that by moves in Theorem \ref{thm:s2} we can always achieve
 one of the 
 following two things: 
 \begin{enumerate}
  \item[1)] $|\Bz|$ is reduced without increasing 
	  $\sum_{i,j}|z^1_{ij}| + \sum_{i,j}|z^{T-1}_{ij}|$;
  \item[2)] $\sum_{i,j}|z^1_{ij}|+\sum_{i,j}|z^{T-1}_{ij}|$ is reduced. 
 \end{enumerate}
 Then either $|\Bz|$ can be reduced to zero
 and we are done, or if we can not further decrease $|\Bz|$ then,
 we can now reduce $\sum_{i,j}|z^1_{ij}| + \sum_{i,j}|z^{T-1}_{ij}|$,
 noting that  $\sum_{i,j}|z^1_{ij}| + \sum_{i,j}|z^{T-1}_{ij}|$ has not increased.
 By this strategy we can either make $|\Bz|=0$ or employ induction on $T$.
 Therefore Theorem \ref{thm:s2} is implied by the above claim.
 
 As a preliminary step consider a path of $\Bx$
 passing through $s=1, t=2$.  If this path ever moves to $s=2$ and then
 comes back to $s=1$ before $T-1$, i.e., if this path contains
 $t\ncolon 12$ and $t'\ncolon 
 21$, $t<t'<T-1$, then we can apply a degree 1 move of the following form 
\[
\begin{tikzpicture}[baseline=-1.2cm]
\foreach \i in {1,2} {   \foreach \j in {1,...,11}    {
 \path (\j,-\i) coordinate (V\i\j);}}
 \foreach \i in {1,2} {   \foreach \j in {1,2,4,5,7,8,10,11}    {
  \fill (V\i\j) circle (1pt);
 }}
 \draw (V11) node [above] {$1$};
 \draw (V14) node [above] {$t'-t+1$};
 \draw (V17) node [above] {$t$};
 \draw (V110) node [above] {$t'$};
 \drbl{11}{22};
 \drbl{22}{23}; 
 \drbl{23}{24}; 
 \drbl{24}{15}; 
 \drwl{17}{28}; 
 \drwl{28}{29};
 \drwl{29}{210};
 \drwl{210}{111};
\end{tikzpicture} 
\]
and reduce %
$\sum_{i,j}|z^1_{ij}|$ without affecting  $\sum_{i,j}|z^{T-1}_{ij}|$.
Therefore in the following we only need
to consider the case that
every path of $\Bx$ passing through $s=1, t=2$ either always stays at $s=1$ until $T-1$
or moves once  to $s=2$ and then stays at $s=2$ thereafter  until $T-1$.

We now check four possible %
sign patterns of $\{z^{T-1}_{ij}\}$.
\begin{center}
1) 
\begin{tikzpicture}[baseline=-15mm]
\foreach \i in {1,2} {   \foreach \j in {1,2}    {
  \path (\j,-\i) coordinate (V\i\j); \fill (V\i\j) circle (2pt); }}
\draw (V11) node [above] {$T-1$};
\draw (V12) node [above] {$T$};
\drw{11}{12}; \drb{21}{12};
\end{tikzpicture} 
\quad
2) 
\begin{tikzpicture}[baseline=-15mm]
\foreach \i in {1,2} {   \foreach \j in {1,2}    {
  \path (\j,-\i) coordinate (V\i\j); \fill (V\i\j) circle (2pt); }}
\draw (V11) node [above] {$T-1$};
\draw (V12) node [above] {$T$};
\drw{11}{22}; \drb{21}{22};
\end{tikzpicture} 
\quad
3) 
\begin{tikzpicture}[baseline=-15mm]
\foreach \i in {1,2} {   \foreach \j in {1,2}    {
  \path (\j,-\i) coordinate (V\i\j); \fill (V\i\j) circle (2pt); }}
\draw (V11) node [above] {$T-1$};
\draw (V12) node [above] {$T$};
\drb{11}{22}; \drw{21}{22};
\end{tikzpicture} 
\quad
4) 
\begin{tikzpicture}[baseline=-15mm]
\foreach \i in {1,2} {   \foreach \j in {1,2}    {
  \path (\j,-\i) coordinate (V\i\j); \fill (V\i\j) circle (2pt); }}
\draw (V11) node [above] {$T-1$};
\draw (V12) node [above] {$T$};
\drb{11}{12}; \drw{21}{12};
\end{tikzpicture}  
\end{center}

\medskip
\noindent {\bf Case 1.}\quad 
\begin{tikzpicture}[baseline=-12mm,scale=0.8]
\foreach \i in {1,2} {   \foreach \j in {1,...,5}    {
  \path (\j,-\i) coordinate (V\i\j);}}
\foreach \i in {1,2} {   \foreach \j in {1,2}    {
 \fill (V\i\j) circle (2pt); }}
\foreach \i in {1,2} {   \foreach \j in {4,5}    {
 \fill (V\i\j) circle (2pt); }}
\draw (V11) node [above] {$1$};
\draw (V12) node [above] {$2$};
\draw (V14) node [above] {$T-1$};
\draw (V15) node [above] {$T$};
\drb{11}{12}; \drw{11}{22};
\drw{14}{15}; \drb{24}{15};
\end{tikzpicture}  \vspace{0.2cm} \\ 
 As shown above, a path in $\Bx$ 
 passing through the node $s=1, t=1,2$ 
 has to either stay at $s=1$ until $T-1$ or has to transit once to $s=2$
 and stay $s=2$ thereafter until $T-1$. 
 Consider the latter case.  
 If %
 a transition from $s=1$ to $s=2$ occurs before $T-1$, then in view of
 $z^{T-1}_{21}>0$, we can assume that the path also contains 
 $T-1:21$ by Lemma \ref{rem:partial-dominance}.
 Then we can apply a degree one move to $\Bx$ and then 
 $\sum_{i,j}|z^1_{ij}|+\sum_{i,j}|z^{T-1}_{ij}|$ is reduced.
 Therefore we only need to consider the case where every path in $\Bx$
 passing through the node $s=1, t=1,2$ move to $s=2$ at $T-1$.  
 By Lemma \ref{rem:partial-dominance}, we can assume that there is no
 transition $t \ncolon 12$, $t=2, \dots, T-1$ in $\Bx$ and therefore
 $z^{t}_{12}\le 0$, for $t=2,\dots,T-1$. 
 Then $z^{T-1}_{12} > 0$ in view of $z^1_{12}<0$.
 Interchanging the roles of $\Bx$ and $\By$ and reversing the time
 direction we similarly have $z^1_{21}<0$.   
 From the fact that $z^1_2=z^T_2=0$, 
 we have $z^1_{22}>0$ and $z^{T-1}_{22}<0$, i.e. $\Bz$ has a sign
 pattern 
 \[
 \begin{tikzpicture}[baseline=-15mm]%
  \foreach \i in {1,2} {   \foreach \j in {1,...,5}    {
  \path (\j,-\i) coordinate (V\i\j);}}
  \foreach \i in {1,2} {   \foreach \j in {1,2}    {
  \fill (V\i\j) circle (2pt); }}
  \foreach \i in {1,2} {   \foreach \j in {4,5}    {
  \fill (V\i\j) circle (2pt); }}
  \draw (V11) node [above] {$1$};
  \draw (V12) node [above] {$2$};
  \draw (V14) node [above] {$T-1$};
  \draw (V15) node [above] {$T$};
  \drb{11}{12}; \drw{11}{22};
  \drb{21}{22}; \drw{21}{12};
  \drw{14}{15}; \drb{24}{15};
  \drb{14}{25}; \drw{24}{25};
 \end{tikzpicture}. 
 \]
 Therefore $|\Bz|$ is reduced by eight by a 2 by 2 swap.

 \vspace{0.2cm}
 \noindent {\bf Case 2.}\quad 
 \begin{tikzpicture}[baseline=-10mm,scale=0.8]
  \foreach \i in {1,2} {   \foreach \j in {1,...,5}    {
  \path (\j,-\i) coordinate (V\i\j);}}
  \foreach \i in {1,2} {   \foreach \j in {1,2}    {
  \fill (V\i\j) circle (2pt); }}
  \foreach \i in {1,2} {   \foreach \j in {4,5}    {
  \fill (V\i\j) circle (2pt); }}
  \draw (V11) node [above] {$1$};
  \draw (V12) node [above] {$2$};
  \draw (V14) node [above] {$T-1$};
  \draw (V15) node [above] {$T$};
  \drb{11}{12}; \drw{11}{22};
  \drw{14}{25}; \drb{24}{25};
 \end{tikzpicture}  \vspace{0.2cm}\\
 Since $z^1_{12}<0, z^{T-1}_{12}<0$, either there exists $2\le t\le T-1$
 such that $z^t_{ij}\ge 2$ or there exist $2\le t< t' \le T-1$ such that
 $z^t_{12}, z^{t'}_{12}>0$. 
 In the former case write $t'=t$.  
 By the symmetry with respect to time reversal of THMC model, we can
 assume that a path in $\Bx$ passing through the node $s=2, t=T, T-1$
 stays at $s=2$ until $t=2$ or transits once to $s=1$ and stays $s=1$
 until $t=2$.  
 By Lemma \ref{rem:partial-dominance}, we can assume that there exist
 two paths passing through $t-1:112$ and $t':122$, respectively: 
 \[
 \begin{tikzpicture}[baseline=-1.2cm,scale=0.7]
  \foreach \i in {1,2} {   \foreach \j in {1,...,5}    {
  \path (\j,-\i) coordinate (V\i\j);  }}
  \draw (V12) node [above] {$t$};
  \draw (V13) node [above] {$t'$};
  \drbl{11}{12}; \drbl{12}{23}; \drbl{13}{24}; \drbl{24}{25};
 \end{tikzpicture}  .
 \]
 Then by subtracting a %
 move of type 4,
 we can widen the difference $t'-t$. 
 By iterating this procedure until either $t=1$ or $t'=T-1$, 
 we can reduce $\sum_{i,j} |z^1_{ij}|+\sum_{i,j} |z^{T-1}_{ij}|$.

 \vspace{0.2cm}
 \noindent {\bf Case 3.}\quad 
 \begin{tikzpicture}[baseline=-10mm,scale=0.8]
  \foreach \i in {1,2} {   \foreach \j in {1,...,5}    {
  \path (\j,-\i) coordinate (V\i\j);}}
  \foreach \i in {1,2} {   \foreach \j in {1,2}    {
  \fill (V\i\j) circle (2pt); }}
  \foreach \i in {1,2} {   \foreach \j in {4,5}    {
  \fill (V\i\j) circle (2pt); }}
  \draw (V11) node [above] {$1$};
  \draw (V12) node [above] {$2$};
  \draw (V14) node [above] {$T-1$};
  \draw (V15) node [above] {$T$};
  \drb{11}{12}; \drw{11}{22};
  \drb{14}{25}; \drw{24}{25};
 \end{tikzpicture}  \vspace{0.2cm}\\
 We can assume $z^{T-1}_{11} \ge 0$ because we have already covered
 Case 1. 
 Then there has to be $2\le t \le T-1$ such that $z^t_{11} < 0$:
 \[
 \begin{tikzpicture}[baseline=-10mm,scale=0.8]
  \foreach \i in {1,2} {   \foreach \j in {1,...,6}    {
  \path (\j,-\i) coordinate (V\i\j);}}
  \foreach \i in {1,2} {   \foreach \j in {1,2}    {
  \fill (V\i\j) circle (2pt); }}
  \foreach \i in {1,2} {   \foreach \j in {5,6}    {
  \fill (V\i\j) circle (2pt); }}
  \fill (V13) circle (2pt);
  \fill (V14) circle (2pt);
  \draw (V11) node [above] {$1$};
  \draw (V12) node [above] {$2$};
  \draw (V13) node [above] {$t$};
  \draw (V15) node [above] {$T-1$};
  \draw (V16) node [above] {$T$};
  \drb{11}{12}; \drw{11}{22};
  \drb{15}{26}; \drw{25}{26};
  \drw{13}{14};
 \end{tikzpicture} \quad .  
 \]
 Also it can be easily seen that $z^1_{21} < 0$ (hence $z^1_{22}>0$) 
 is reduced to Case 1 by a symmetry consideration.  Therefore
 we can assume $z^1_{21} \ge 0$.  Then
 \[
 z^2_1 = z^1_{11} + z^1_{21} > 0, 
\]
 i.e., %
 $\Bx$ dominates $\By$ at the node $s=1, t=2$.
 Then  $z^2_2 < 0$ because  $z^2_1 + z^2_2=0$.
 Similar consideration shows that  $z^{T-1}_1 > 0$ and $z^{T-1}_2 < 0$.
 Suppose that a path of $\By$ passing %
 $1\ncolon 12$ arrives at $s=1$ at time $t$. 
 By Lemma \ref{rem:partial-dominance} we can assume that the path stays
 at $s=1$ at time $t+1$.  
 Denote the path by $w=(12s_3 \cdots s_{t-1}11s_{t+2}\cdots s_T)$. 
 Let $w^\prime = (112s_3 \cdots s_{t-1}1s_{t+2} \cdots s_T)$.
 Then the difference of $w$ and $w^\prime$ forms a degree one
 move with a sign pattern
 \[
 \begin{tikzpicture}[baseline=-10mm,scale=0.8]
  \foreach \i in {1,2} {   \foreach \j in {1,...,2}    {
  \path (\j,-\i) coordinate (V\i\j);}}
  \foreach \i in {1,2} {   \foreach \j in {1,2}    {
  \fill (V\i\j) circle (2pt); }}
  \draw (V11) node [above] {$1$};
  \draw (V12) node [above] {$2$};
  \drw{11}{12}; \drb{11}{22};
 \end{tikzpicture} \quad .  \vspace{0.2cm}\\
 \]
 By adding it to $\Bx$, we can reduce 
 $\sum_{i,j}|z^1_{ij}|$ by two without affecting
 $\sum_{i,j}|z^{T-1}_{ij}|$.  
 Therefore we only need to consider the case where there is no path of
 $\By$ passing the node $s=2, t=2$ and arriving  at $s=1$ at time $t$. 

 Since $z^2_1>0$, there has to be some $t'\le t$
 such that $z^{t'}_{12} > 0$.
 Now assume that %
 $t^\prime < T-2$ 
 and choose the minimum $t^\prime$ %
 with $z^{t^\prime}_{12}>0$.
 Then a path in $\Bx$ passing through $(s,t)=(1,2)$ and $t^\prime:12$   
 has to stay at $s=2$ from $t'+1$ to $T-1$.
 Recall $0 < z^{T-1}_1 = z^{T-2}_{11} + z^{T-2}_{21}$ and 
 suppose that $z^{T-2}_{11} > 0$.
 Then we can assume that there exists a path passing through
 $T-2:112$.
 Therefore we can apply type 4 move 
 \[
 \begin{tikzpicture}[baseline=-1.0cm,scale=0.8]
  \foreach \i in {1,2} {   \foreach \j in {1,...,6}    {
  \path (\j,-\i) coordinate (V\i\j); }}
  \fill (V11) circle (1pt);
  \fill (V12) circle (1pt);
  \fill (V21) circle (1pt);
  \fill (V22) circle (1pt);
  \fill (V23) circle (1pt);
  \fill (V14) circle (1pt);
  \fill (V15) circle (1pt);
  \fill (V25) circle (1pt);
  \fill (V16) circle (1pt);
  \fill (V26) circle (1pt);
  \draw (V11) node [above] {$t-1$};
  \draw (V12) node [above] {$t$};
  \draw (V14) node [above] {$T-2$};
  \draw (V16) node [above] {$T$};
  \drbl{11}{12}; \drbl{12}{23}; \drbl{14}{25}; \drbl{25}{26};
  \drwl{11}{22}; \drwl{22}{23}; \drwl{14}{15}; \drwl{15}{26};
 \end{tikzpicture}  .
 \]
 and we can reduce $\sum_{i,j}|z^{T-1}_{ij}|$ without affecting
 $\sum_{i,j}|z^1_{ij}|$.  
 Suppose that $z^{T-2}_{21} > 0$.
 Then we can assume that there exists a path passing through
 $T-3:2212$ and hence we can apply a degree one move
 \[
 \begin{tikzpicture}[baseline=-1.2cm,scale=0.8]
  \foreach \i in {1,2} {   \foreach \j in {1,...,4}    {
  \path (\j,-\i) coordinate (V\i\j);  \fill (V\i\j) circle (1pt); }}
  \draw (V11) node [above] {$T-2$};
  \draw (V14) node [above] {$T$};
  \draw (V21)--(V12)--(V23)--(V24);
  \draw [dotted] (V21)--(V22)--(V13)--(V24);
 \end{tikzpicture}
 \]
 to reduce $\sum_{i,j}|z^{T-1}_{ij}|$ without affecting
 $\sum_{i,j}|z^1_{ij}|$.  
 Therefore we only need to consider the case of $t'=t=T-2$.
 Since $z^{T-1}_1 >0$ and $z^{T-2}_{11}<0$, we have 
 $z^{T-2}_{21}>0$. 
 In the same way we also have $z^{T-2}_{12}<0$. 
 Hence $\Bz$ has a sign pattern
 \[
 \begin{tikzpicture}[baseline=-10mm,scale=0.8]
  \foreach \i in {1,2} {   \foreach \j in {1,2}{
  \path (\j,-\i) coordinate (V\i\j);
    \fill (V\i\j) circle (2pt);}}
  \draw (V11) node [above] {{\scriptsize$T-2$}};
  \draw (V12) node [above] {{\scriptsize \ $T-1$}};
  \drb{11}{22}; \drb{21}{12};
  \drw{11}{12}; \drw{21}{22};
 \end{tikzpicture}\ .
 \]
 By symmetry we can easily see that 
 $\Bz$ has a sign pattern
\[
 \begin{tikzpicture}[baseline=-12mm,scale=0.8]
  \foreach \i in {1,2} {  \foreach \j in {1,...,5}{
  \path (\j,-\i) coordinate (V\i\j);}}
  \foreach \i in {1,2} {   \foreach \j in {1,2}    {
  \fill (V\i\j) circle (2pt); }}
  \foreach \i in {1,2} {   \foreach \j in {4,5}    {
  \fill (V\i\j) circle (2pt); }}
  \draw (V11) node [above] {{$2$}};
  \draw (V12) node [above] {{$3$}};
  \drw{11}{22}; \drw{21}{12};
  \drb{11}{12}; \drb{21}{22};
  \draw (V14) node [above] {{\scriptsize $T-2$}};
  \draw (V15) node [above] {{\scriptsize $T-1$}};
  \drw{14}{15}; \drw{24}{25};
  \drb{14}{25}; \drb{24}{15};
 \end{tikzpicture}\ .  
\] 
 Then we can decrease $|\Bz|$ by a 2 by 2 swap.

 \vspace{0.2cm}
 \noindent {\bf Case 4.}\quad 
 \begin{tikzpicture}[baseline=-10mm,scale=0.8]
  \foreach \i in {1,2} {   \foreach \j in {1,...,5}    {
  \path (\j,-\i) coordinate (V\i\j);}}
  \foreach \i in {1,2} {   \foreach \j in {1,2}    {
  \fill (V\i\j) circle (2pt); }}
  \foreach \i in {1,2} {   \foreach \j in {4,5}    {
  \fill (V\i\j) circle (2pt); }}
  \draw (V11) node [above] {$1$};
  \draw (V12) node [above] {$2$};
  \draw (V14) node [above] {$T-1$};
  \draw (V15) node [above] {$T$};
  \drb{11}{12}; \drw{11}{22};
  \drb{14}{15}; \drw{24}{15};
 \end{tikzpicture} \vspace{0.2cm}\\
 By a consideration as in the previous case %
 there exists $t$ such that $z^t_{11}<0$: 
 \[
 \begin{tikzpicture}[baseline=-10mm,scale=0.8]
  \foreach \i in {1,2} {   \foreach \j in {1,...,6}    {
  \path (\j,-\i) coordinate (V\i\j);}}
  \foreach \i in {1,2} {   \foreach \j in {1,2}    {
  \fill (V\i\j) circle (2pt); }}
  \foreach \i in {1,2} {   \foreach \j in {5,6}    {
  \fill (V\i\j) circle (2pt); }}
  \fill (V13) circle (2pt);
  \fill (V14) circle (2pt);
  \draw (V11) node [above] {$1$};
  \draw (V12) node [above] {$2$};
  \draw (V13) node [above] {$t$};
  \draw (V15) node [above] {$T-1$};
  \draw (V16) node [above] {$T$};
  \drb{11}{12}; \drw{11}{22};
  \drb{15}{16}; \drw{25}{16};
  \drw{13}{14};
 \end{tikzpicture} \quad .  \\
 \]
 As in the previous case there should exist (the minimum) $t' \le t$ with   
 $z^{t'}_{12} >0$.
 By symmetry there also exists (the maximum) $t''\ge t+1$ with
 $z^{t''}_{21}>0$. 
 Then we can apply a degree 1 move 
 \[
 \begin{tikzpicture}[baseline=-10mm,scale=0.8]
  \foreach \i in {1,2} {   \foreach \j in {1,...,13}    {
  \path (\j,-\i) coordinate (V\i\j);}}
  \fill (V11) circle (1pt);
  \fill (V12) circle (1pt);
  \fill (V22) circle (1pt);
  \fill (V14) circle (1pt);
  \fill (V24) circle (1pt);
  \fill (V15) circle (1pt);
  \fill (V16) circle (1pt);
  \fill (V17) circle (1pt);
  \fill (V18) circle (1pt);
  \fill (V29) circle (1pt);
  \fill (V111) circle (1pt);
  \fill (V211) circle (1pt);
  \fill (V112) circle (1pt);
  \fill (V113) circle (1pt);
  \draw (V11) node [above] {$1$};
  \draw (V12) node [above] {$2$};
  \draw (V14) node [above] {$t''-t'+1$};
  \draw (V18) node [above] {$t'$};
  \draw (V111) node [above] {$t''$};
  \draw (V11)--(V22)--(V23)--(V24)--(V15)--(V16);
  \draw[dotted] (V17)--(V18)--(V29)--(V210)--(V211)--(V112)--(V113);
 \end{tikzpicture} \quad .  \\
\]
 to decrease
 $\sum_{i,j} |z^1_{ij}|$.
This completes the proof of Theorem \ref{thm:s2}.
\end{proof}

\section{Markov basis for the case of length of three}
\label{sec:three-period}

In this section we give an explicit form of a Markov basis for
the case of $T=3$. The number of states  $S=|\cS|$ is arbitrary.

Let $i_1, \dots, i_m$, $m\le S$,  be distinct elements of $\cS$. Similarly
let $j_1, \dots, j_m$, $m\le S$,  be distinct elements of $\cS$.
In the $m$ by $m$ permutation of (\ref{eq:m-by-m-swap}) let $t=1, t'=2$ and 
\[
\sigma(1)=m, \sigma(2)=1, \dots, \sigma(k)=m-1.
\]
Then the resulting set of $m$ by $m$ permutation is denoted by
\begin{align}
 &{\cal Z}(i_1,\dots,i_m; j_1,\dots,j_m)\ :  \nonumber \\ 
 & \qquad\qquad \{1\ncolon i_1 j_1, 1\ncolon i_2 j_2, \dots , 
 1\ncolon i_m j_m\}\ \leftrightarrow\  
 \{2\ncolon i_1 j_m, 2\ncolon i_2 j_1, \dots, 2 \ncolon i_m j_{m-1}\}.
 \label{eq:moves-T3}
\end{align}
A move graph of a permutation for $S=6$ with $m=4$, $(i_1, i_2, i_3,
i_4)=(1,2,4,5)$, $(j_1, j_2, j_3, j_4)=(1,3,5,6)$ is depicted in Figure  
\ref{fig:permutation}.
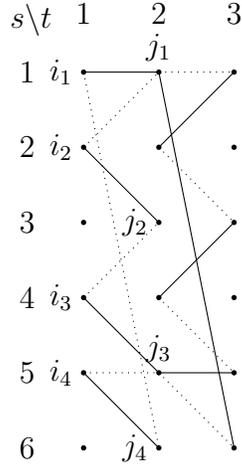
\begin{figure}[ht]
 \begin{center}
  \begin{tikzpicture}
   \foreach \i in {1,...,6} {   \foreach \j in {1,...,3}    {
  \path (\j,-\i) coordinate (V\i\j);  \fill (V\i\j) circle (1pt); }}
   \foreach \i in {1,...,6} {\draw (0.5,-\i) node [left] {\i};}
   \foreach \j in {1,...,3} {\draw (\j,-0.5) node [above] {\j};}
   \draw (0.3,-0.3) node {$s\backslash t$};
   \drbl{11}{12};\drwl{12}{13};
   \drwl{21}{12};\drbl{22}{13};
   \drbl{21}{32};\drwl{22}{33};
   \drwl{41}{32};\drbl{42}{33};
   \drbl{41}{52};\drwl{42}{53};
   \drwl{51}{52};\drbl{52}{53};
   \drbl{51}{62};\drwl{52}{63};
   \drwl{11}{62};\drbl{12}{63};
   \draw (V11) node [left] {$i_1$};
   \draw (V21) node [left] {$i_2$};
   \draw (V41) node [left] {$i_3$};
   \draw (V51) node [left] {$i_4$};
   \draw (V12) node [above] {$j_1$};
   \draw (V32) node [left] {$j_2$};
   \draw (V52) node [above] {$j_3$};
   \draw (V62) node [left] {$j_4$};
  \end{tikzpicture}
 \end{center}
\caption{A typical permutation}
\label{fig:permutation}
\end{figure}
Now we state the following theorem on a Markov basis for the case of $T=3$.

\begin{theorem}
\label{thm:T3}
 A Markov basis for THMC model with $T=3$ is given by the set of
 crossing path swappings in (\ref{eq:crossing}) and 
 the set of $m \times m$ permutations  
 ${\cal Z}(i_1,\dots,i_m;\allowbreak j_1,\dots,j_m)$ in~(\ref{eq:moves-T3}), 
 where $m=2,\dots, S$, $i_1,\dots, i_m$ are
 distinct, and $j_1, \dots, j_m$ are distinct.
\end{theorem}

% As discussed in Section \ref{subsec:preparations}, an $m$ by $m$
% permutation corresponds to a collection of moves.  
% Theorem \ref{thm:T3} states that a Markov basis is formed by the union
% of these collections of moves for the permutations.  
For proving Theorem \ref{thm:T3} we need more
detailed consideration of these moves.

Let
\[
\cI=\{i_1, \dots, i_m\}, \quad \cJ=\{j_1, \dots, j_m\}.
\]
For each $j\in \cI\cap \cJ$, there exist unique indices 
$i(j), i'(j)\in \cI$, $k(j), k'(j)\in \cJ$, such that 
\begin{align*}
&1\ncolon i(j)j  \, \in \, \{1\ncolon i_1 j_1, 1\ncolon i_2 j_2, \dots, 1\ncolon i_m j_m\}, \\
&2\ncolon jk'(j) \, \in \, \{2\ncolon i_1 j_1, 2\ncolon i_2 j_2, \dots, 2\ncolon i_m j_m\}, \\
&1\ncolon i'(j)j \, \in \, \{1\ncolon i_1 j_m, 1\ncolon i_2 j_1,\dots, 1 \ncolon i_m j_{m-1}\},\\
&2\ncolon jk(j)  \, \in \, \{2\ncolon i_1 j_m, 2\ncolon i_2 j_1, \dots, 2 \ncolon i_m j_{m-1}\}.
\end{align*}
For example consider $j=1$ in Figure \ref{fig:permutation}.
Then $i(1)1k(1)=116$ is a solid path and $i'(1)1 k'(1)=211$ is
a dotted path.  
Next, for each $j\in \cI\setminus \cJ$, there exist unique indices
$k(j),k'(j)$ such that
\[
2\ncolon jk(j)\in  \{2\ncolon i_1 j_m, 2\ncolon i_2 j_1, \dots, 2 \ncolon i_m j_{m-1}\},\quad
2\ncolon jk'(j)\in \{2\ncolon i_1 j_1, 2\ncolon i_2 j_2, \dots, 2\ncolon i_m j_m\}.
\]
For $j=2$ in Figure \ref{fig:permutation}, $k(j)=1, k'(j)=3$, 
$2\ncolon 21$ is a solid edge and $2\ncolon 23$ is a dotted edge.
Finally for each $j\in \cJ\setminus \cI$
there exist unique indices
$i(j),i'(j)$ such that
\[
1\ncolon i(j)j \in  \{1\ncolon i_1 j_1, 1\ncolon i_2 j_2, \dots, 1\ncolon i_m j_m\}, \quad
1\ncolon i'(j)j \in \{1\ncolon i_1 j_m, 1\ncolon i_2 j_1, \dots, 1\ncolon i_m j_{m-1}\}.
\]
For $j=3$ in Figure \ref{fig:permutation}, $i(j)=2, i'(j)=4$, 
$1\ncolon 23$ is a solid edge and $1\ncolon 43$ is a dotted edge.
Define the set of paths $W_1$ and $W_2$ as follows,
\begin{align*}
W_1 &= \{i(j)jk(j), j\in \cI \cap \cJ\} \cup
 \{s_j j k(j), j\in \cI\setminus \cJ\} \cup
 \{i(j)j \tilde s_j, j\in \cJ\setminus \cI \} \nonumber\\
W_2 &=
\{i'(j)jk'(j), j\in \cI\cap \cI\}\cup
\{s_j j k'(j), j\in \cI\setminus \cJ\} \cup
\{i'(j)j \tilde s_j, j\in \cJ\setminus \cI \},
\end{align*}
where $s_j,j\in \cI\setminus \cJ$ 
and $\tilde s_j, j\in \cJ\setminus \cI$ are arbitrary states.
Then we consider a set of moves 
$\bar Z(i_1,\dots,i_m;j_1,\dots,j_m) 
= \{\bar z(\omega \mid i_1,\dots,i_m;j_1,\dots,j_m) \mid \omega \in
\cS^T\}$,  
\begin{align}
 \label{eq:m-by-m-minimal}
 \bar z(\omega \mid i_1,\dots,i_m;j_1,\dots,j_m) = \left\{
 \begin{array}{rl}
  1 & \text{if } \omega \in W_1\\
  -1 & \text{if } \omega \in W_2\\
  0 & \text{otherwise}\\
 \end{array}
 \right.
\end{align}
$\bar Z(i_1,\dots,i_m;j_1,\dots,j_m)$ is easily shown to be an $m$ by
$m$ permutation.  
Now we have the following proposition, which implies Theorem \ref{thm:T3}.

\begin{proposition}
\label{prop:T3-strict}
A Markov basis for THMC model for $T=3$ is given by the set of
crossing path swappings in (\ref{eq:crossing}) and moves of 
$\bar Z(i_1,\dots,i_m;j_1,\dots,j_m)$ of the form  (\ref{eq:m-by-m-minimal}),
where $m=2,\dots, S$, $i_1,\dots, i_m$ are distinct, and
$j_1, \dots, j_m$ are distinct.
\end{proposition}
\begin{proof}

 For $T=3$, (\ref{eq:zsum2T1}) implies that node frequencies are common
for two elements $\Bx$ and $\By$ of the same fiber.  Therefore the
 edge-sign pattern 
graph $G$ has the following property (P1):
\begin{quote}
If $z^1_{ij}>0$\ \ ($<0$), there exists $i'$ such that $z^1_{i'j}<0$ \ \  ($>0$).\\
If $z^2_{ij}>0$\ \ ($<0$), there exists $j'$ such that $z^2_{ij'}<0$\ \  ($>0$).
\end{quote}
Clearly similar property holds for initial nodes at $t=1$ and terminal nodes
at $t=3$.  Also for $T=3$ we have $z^1_{ij}+z^2_{ij}=0$ and therefore
$z^1_{ij}>0$ \ ($<0$) if and only if $z^2_{ij}<0$ \ ($>0$).
This means that $G$ for $t=2,3$ is a copy of $G$ for $t=1,2$ with
sings of edges reversed.

Suppose that $\Bx$ and $\By$ are not edge-wise equivalent.  Then by the
above properties $G$ has some positive edge $1\ncolon i_1 j_1$.  We follow the
edge from $t=1$ to $t=2$.  Then by (P1), we can 
follow a negative edge $1\ncolon i_2 j_1$ back to $t=1$.  
Then we can again follow a positive edge $1\ncolon i_2 j_2$ to $t=2$, etc.
By continuing this, the same node is eventually visited twice and a loop is formed.
By considering the shortest loop, it can be easily shown 
that there exists the following simple loop of $G$ for $t=1,2$:

\begin{quote}
There are distinct indices $i_1,\dots,i_m$ and
distinct indices $j_1, \dots, j_m$ such that
$1\ncolon i_1 j_1, 1\ncolon i_2 j_2, \dots, 1\ncolon i_m j_m$ are positive edges of $G$ and
$1\ncolon i_1 j_m, 1\ncolon i_2 j_1, \dots 1\ncolon i_m j_{m-1}$ 
are negative edges of $G$.
\end{quote}

Now we employ the distance reduction argument of
\citet{takemura-aoki-2005bernoulli}.
As above let $\cI=\{i_1,\dots,i_m\}$ and $\cJ=\{j_1, \dots, j_m\}$.
By Lemma \ref{rem:partial-dominance} we can assume that
$x(i(j)jk(j))>0$  for $j\in \cI\cap \cJ$.
For $j\in \cI\setminus \cJ$ we can find a path  $\omega=s_j j k(j)$
with $x(\omega)>0$.  Similarly for $j\in \cJ\setminus \cI$ 
we can find a path $\omega=i(j)j \tilde s_j$ such that
$x(\omega)>0$.  
Now we can subtract %
a move in (\ref{eq:m-by-m-minimal}) from $\Bx$. 
 This reduces %
 $\vert z \vert$ by $4m$.
 This completes the proof of Proposition \ref{prop:T3-strict}.
\end{proof}

\section{A numerical example}
\label{sec:numerical}
In this section we give a numerical example for testing
THMC model against the non-homogeneous Markov chain model using
our Markov basis.  

\begin{table}[htbp]
\centering
 \label{tab1}
 \caption{Marijuana use data}
 \begin{tabular}{c|ccc|ccc|ccc}\hline
  & \multicolumn{9}{c}{$1978$}\\ 
  & \multicolumn{3}{c}{1} & \multicolumn{3}{c}{2}
  & \multicolumn{3}{c}{3}\\ \cline{2-10}
  $1977$ & \multicolumn{3}{c|}{$1979$} & \multicolumn{3}{c|}{$1979$} 
  & \multicolumn{3}{c}{$1979$}\\ 
    & 1  & 2 & 3 & 1 & 2  & 3 & 1 & 2 & 3\\ \hline
  1 & 76 & 6 & 6 & 4 & 12 & 1 & 0 & 0 & 1 \\
  2 & 3  & 0 & 0 & 1 & 1  & 2 & 1 & 0 & 2\\
  3 & 0  & 0 & 0 & 0 & 0  & 0 & 0 & 2 & 2\\ \hline
 \end{tabular}
\end{table}

Table 1 refers to a longitudinal data from 1977 to 1979 on 
marijuana use of 120 female respondents who were age 14 in 1977 (e.g
\cite{vermunt-etal}).  
The degrees of dependence are categorized into the following three
levels, 1.\ never use; 2.\ no more than once a month; 3.\ more than once a
month.  

We examine the goodness-of-fit of $H_0$ with 
$(S,T)=(3,3)$ by testing $\bar H_0$.
We use Pearson's $\chi^2$ as a test statistic.
We have $\chi^2=11.533$.

We computed the exact distribution of $\chi^2$ via MCMC with a Markov
basis derived in Theorem \ref{thm:T3}. 
We sampled 100,000 tables after 50,000 burn-in steps.
Figure \ref{fig:2} represents the histogram of the sampling distribution
of $\chi^2$ and the solid line in the figure is the density function of
the asymptotic $\chi^2$ distribution with degrees of freedom 4. 
The asymptotic $p$-value and the exact $p$-value are 
$0.0212$ and $0.0184$ and THMC model is rejected.
For this example the homogeneity of Markov chain model is rejected,
since THMC model is rejected.  
However, note that the approximation by $\chi^2$ distribution is not
very good for this data.  
This is probably due to small frequencies in the data set.

\begin{figure}[htbp]
 \centering
 \label{fig:2}
 \includegraphics[scale=0.4]{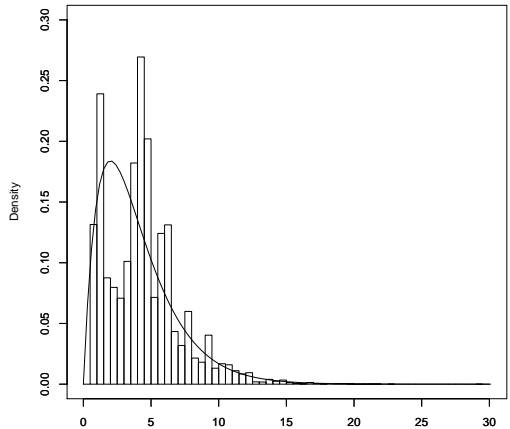}
 \caption{Sampling distribution of $L$ via MCMC}
\end{figure}

\section{Discussion and a conjecture for general case}
\label{sec:discussion}

From the viewpoint of application, it is clearly desirable to obtain
Markov bases for general $S$ and $T$.  By 4ti2 we can obtain a Markov
basis for $S=3$ and $T=5$.  However currently this case seems to be the
largest case which can be handled by a software.

Even if obtaining a Markov basis is difficult, it is of theoretical
interest to know some properties of Markov bases for general $S$ and
$T$.  Note that the description of Markov basis for $S=2$ is common for
all $T\ge 4$.  This is very similar to the notion of Markov complexity
(\cite{hosten-sullivant-2007}, \cite{santos-sturmfels-2003}, \cite{aoki-takemura-2003anz})
and we believe that there exists some bound of complexity of Markov
basis in $T$ for a given $S$.

We have done extensive investigation of 
edge-sign pattern graphs $G$ for the case  $S=3$ and $T=4$.
We will discuss these edge-sign pattern graphs in \cite{takemura-kudo-hara}.
It seems that the following notion 
of {\em extended simple loop} of the edge-sign pattern  graph $G$
is important.
Here we give only a rough definition of an extended simple loop. 
An extended simple loop is a loop, such that when we are moving
towards the future we follow positive edges and when we are moving
towards the past we follow negative edges of $G$.  Also we require that
each node is passed at most once. 
An important example of an extended simple loop
for $S=3$ and $T=4$ is depicted as follows.
\begin{center}
\begin{tikzpicture}
\mytkzini
\drb{11}{32}; \drb{21}{12}; \drb{12}{23}; \drb{23}{34};
\drw{11}{22}; \drw{21}{32}; \drw{22}{13}; \drw{13}{34};
\end{tikzpicture}
\end{center}
We have used the term ``extended'' to indicate that the lengths of
partial paths may be greater than one.  If all the edges are of length
one (in this case the loop involves only two time points $t$ and
$t+1$), we call it a {\em simple loop}.

Note that for the case of $S=2$, a Markov basis consists of moves,
which are sum of at most two extended simple loops.  For the case of
$S=3$ and $T=4$ we checked that a Markov basis consists of moves,
which are sums of at most three extended simple loops.  Therefore our
conjecture at this point is that for each  $S$, there exists $k_S$,
such that a Markov basis
consists of moves which are sums of at most $k_S$ extended simple
loops.  A stronger conjecture is $k_S=S$.

\bibliographystyle{plainnat}
\bibliography{homogeneous1}

\end{document}